\documentclass[a4paper,12pt]{article}

\usepackage{amsfonts}
\usepackage{amscd,color}
\usepackage{amsmath,amsfonts,amssymb,amscd}
\usepackage{indentfirst,graphicx,epsfig}
\usepackage{graphicx}
\input{epsf}
\usepackage{epstopdf}
\usepackage{caption}
\usepackage{ifpdf}
\usepackage{fullpage}
\usepackage{epsfig}
\usepackage{latexsym}
\usepackage{caption}
\usepackage{amsmath,amsthm}
\usepackage{amsfonts}
\usepackage{amssymb}
\usepackage{graphicx}
\usepackage{graphics}
\usepackage{bm}
\usepackage{color}
\usepackage{subfigure}
\graphicspath{{\figs}}

\makeatletter

\newcommand{\Rmnum}[1]{\expandafter\@slowromancap\romannumeral #1@}
\makeatother
\makeatletter
\let\@fnsymbol\@arabic
\makeatother

\begin{document}
\newtheorem{theorem}{Theorem}[section]
\newtheorem{observation}[theorem]{Observation}
\newtheorem{corollary}[theorem]{Corollary}
\newtheorem{algorithm}[theorem]{Algorithm}
\newtheorem{problem}[theorem]{Problem}
\newtheorem{question}[theorem]{Question}
\newtheorem{lemma}[theorem]{Lemma}
\newtheorem{proposition}[theorem]{Proposition}

\newtheorem{definition}[theorem]{Definition}
\newtheorem{guess}[theorem]{Conjecture}
\newtheorem{claim}[theorem]{Claim}
\newtheorem{example}[theorem]{Example}
\makeatletter
  \newcommand\figcaption{\def\@captype{figure}\caption}
  \newcommand\tabcaption{\def\@captype{table}\caption}
\makeatother

\newtheorem{acknowledgement}[theorem]{Acknowledgement}

\newtheorem{axiom}[theorem]{Axiom}
\newtheorem{case}[theorem]{Case}
\newtheorem{conclusion}[theorem]{Conclusion}
\newtheorem{condition}[theorem]{Condition}
\newtheorem{conjecture}[theorem]{Conjecture}
\newtheorem{criterion}[theorem]{Criterion}
\newtheorem{exercise}[theorem]{Exercise}
\newtheorem{notation}[theorem]{Notation}
\newtheorem{solution}[theorem]{Solution}
\newtheorem{summary}[theorem]{Summary}
\newtheorem{fact}[theorem]{Fact}

\newcommand{\pp}{{\it p.}}
\newcommand{\de}{\em}
\newcommand{\mad}{\rm mad}

\newcommand*{\QEDA}{\hfill\ensuremath{\blacksquare}}
\newcommand*{\QEDB}{\hfill\ensuremath{\square}}

\newcommand{\qf}{Q({\cal F},s)}
\newcommand{\qff}{Q({\cal F}',s)}
\newcommand{\qfff}{Q({\cal F}'',s)}
\newcommand{\f}{{\cal F}}
\newcommand{\ff}{{\cal F}'}
\newcommand{\fff}{{\cal F}''}
\newcommand{\fs}{{\cal F},s}
\newcommand{\g}{\gamma}
\newcommand{\wrt}{with respect to }

\title{Monochromatic disconnection of graphs\footnote{Supported by NSFC No.11871034, 11531011 and NSFQH No.2017-ZJ-790.}}

\renewcommand{\thefootnote}{\arabic{footnote}}

\author{\small Ping Li$^1$, ~ Xueliang Li$^{1,2}$\\
\small $^1$Center for Combinatorics and LPMC\\
\small  Nankai University\\
\small Tianjin 300071, China\\
\small Email: qdli\underline{ }ping@163.com, ~ lxl@nankai.edu.cn\\
\small $^2$School of Mathematics and Statistics\\
\small Qinghai Normal University\\
\small Xining, Qinghai 810008, China\\
}

\date{}
\maketitle

\begin{abstract}
For an edge-colored graph $G$, we call an edge-cut $M$ of $G$ monochromatic if the edges of $M$ are colored with a same color.
The graph $G$ is called monochromatically disconnected if any two distinct vertices of $G$ are separated by a monochromatic edge-cut.
For a connected graph $G$, the monochromatic disconnection number, denoted by $md(G)$, of $G$ is the maximum number of colors
that are needed in order to make $G$ monochromatically disconnected. We will show that almost all graphs have monochromatic
disconnection numbers equal to 1. We also obtain the Nordhaus-Gaddum-type results for $md(G)$.\\[2mm]
{\bf Keywords:} monochromatic edge-cut, monochromatic disconnection number, Nordhaus-Gaddum-type results.\\[2mm]
{\bf AMS subject classification (2010)}: 05C15, 05C40, 05C35.
\end{abstract}

\baselineskip16pt

\section{Introduction}
Let $G$ be a graph and let $V(G)$, $E(G)$ denote the vertex set and the edge set of $G$, respectively.
Let $|G|$ (also $v(G)$) denote the number of vertices of $G$, called the order of $G$. If there is no confusion,
we use $n$ and $m$ to denote, respectively, the number of vertices and the number of edges of
a graph, throughout this paper.
For $v\in V(G)$, let $d_G(v)$ denote the degree of $v$. We call a vertex $v$ a {\em $t$-degree} vertex of $G$ if $d_G(v)=t$.
Let $\delta(G)$ and $\Delta(G)$ denote the minimum and maximum degree of $G$, respectively. Sometimes,
we also use $\Delta$ to denote a triangle. We use $\overline{G}$ to denote the complement graph of $G$.
Let $S$ and $F$ be a vertex set and an edge set of $G$, respectively. $G-S$ is a graph obtained from $G$
by deleting the vertices of $S$ together with the edges incident with vertices of $S$. $G-F$ is a graph whose
vertex set is $V(G)$ and edge set is $E(G)-F$. Let $G[S]$ and $G[F]$ be the vertex-induced and edge-induced subgraphs of $G$,
respectively, by $S$ and $F$. The distance of $u,v$ in $G$ is denoted by $d_G(u,v)$.
For all other terminology and notation not defined here we follow Bondy and Murty \cite{B}.

Throughout this paper, we use $K_n, K_{n_1,n_2}$ and $C_n$ to denote a complete graph, a complete  bipartite graph, a cycle of order $n$, respectively.
Let $K_n^-$ be the graph obtained from $K_n$ by deleting an arbitrary edge. $K_3$ is also called a {\em triangle}. We call a cycle $C$ a {\em $t$-cycle}
if $|C|=t$. We use $[r]$ to denote the set $\{1,2,\cdots,r\}$ of positive integers.

For a graph $G$, let $\Gamma: E(G)\rightarrow [r]$ be an {\em edge-coloring} of $G$ that allows a
same color to be assigned to adjacent edges. For an edge $e$ of $G$, we use $\Gamma(e)$ to denote the color of $e$.
If $H$ is a subgraph of $G$, we also use $\Gamma(H)$ to denote the set of colors on edges of $H$ and use $|\Gamma(H)|$
to denote the number of colors in $\Gamma(H)$. An edge-coloring $\Gamma$ of $G$ is {\em trivial} if $|\Gamma(G)|=1$,
otherwise, it is {\em nontrivial}.

For two vertices $u$ and $v$ of an edge-colored graph $G$, a {\em rainbow $uv$-path} is a path of $G$ between $u$ and $v$ such that the edges
on the path are colored pairwise differently, and $G$ is {\em rainbow connected} if any two distinct vertices of $G$ are connected by a rainbow path.
An edge-coloring $\Gamma$ of $G$ is a {\em rainbow connection coloring} if it makes $G$ rainbow connected. For a connected graph $G$, the {\em rainbow connection number}
of $G$, denoted by $rc(G)$, is the minimum number of colors that are needed in order to make $G$ rainbow connected.
The notion rainbow connection coloring was introduced by Chartrand et al. in \cite{CJMZ}.

An {\em edge-cut} of a connected graph $G$ is an edge set $F$ such that $G-F$ is not connected. For an edge-colored graph $G$,
we call an edge-cut $R$ a {\em rainbow edge-cut} if the edges of $R$ are colored pairwise different.
For two vertices $u,v$ of $G$, a {\em rainbow $uv$-cut} is a rainbow edge-cut that separates $u$ and $v$.
An edge-colored graph $G$ is {\em rainbow disconnected} if any two vertices of $G$ has a rainbow cut separating them.
An edge-coloring of $G$ is a {\em rainbow disconnection coloring} if it makes $G$ rainbow disconnected.
For a connected graph $G$, the {\em rainbow disconnection number} of $G$, denoted by $rd(G)$, is the minimum number of colors that are needed in
order to make $G$ rainbow disconnected. The notion rainbow connection coloring was introduced by Chartrand et al. in \cite{CDHHZ}.

Contrary to the concepts for rainbow connection and disconnection, monochromatic versions of these concepts naturally appeared, as the other extremal.
For two vertices $u$ and $v$ of an edge-colored graph $G$, a {\em monochromatic $uv$-path} is a $uv$-path of $G$ whose edges are colored with a same color,
and $G$ is {\em monochromatically connected} if any two distinct vertices of $G$ are connected by a monochromatic path.
An edge-coloring $\Gamma$ of $G$ is a {\em monochromatic connection coloring} if it makes $G$ monochromatically connected.
For a connected graph $G$, the {\em monochromatic connection number} of $G$, denoted by $mc(G)$, is the maximum number of colors that are needed in
order to make $G$ monochromatically connected. The notion monochromatic connection coloring was introduced by Caro and Yuster in \cite{CY}.

As a counterpart of the rainbow disconnection coloring and a similar object of the monochromatic connection coloring,
we now introduce the notion of monochromatic disconnection coloring of a graph. For an edge-colored graph $G$, we call an edge-cut $M$
a {\em monochromatic edge-cut} if the edges of $M$ are colored with a same color. For two vertices $u,v$ of $G$, a {\em monochromatic $uv$-cut}
is a monochromatic edge-cut that separates $u$ and $v$. An edge-colored graph $G$ is {\em monochromatically disconnected} if any two vertices of $G$
has a monochromatic cut separating them. An edge-coloring of $G$ is a {\em monochromatic disconnection coloring} ($MD$-coloring for short)
if it makes $G$ monochromatically disconnected. For a connected graph $G$, the {\em monochromatic disconnection number} of $G$, denoted by $md(G)$,
is the maximum number of colors that are needed in order to make $G$ monochromatically disconnected. An {\em extremal MD-coloring} of $G$ is an $MD$-coloring
that uses $md(G)$ colors. If $H$ is a subgraph of $G$ and $\Gamma$ is an edge-coloring of $G$, we call $\Gamma$ an edge-coloring {\em restricted} on $H$.

As we know that there are two ways to study the connectivity of a graph, one way is by using paths and the other is by using cuts.
Both rainbow connection and monochromatic connection provide ways to study the colored connectivity of graph by colored paths. However, both rainbow disconnection
and monochromatic disconnection can provide ways to study the colored connectivity of graph by colored cuts. All these parameters or numbers coming from studying
the colored connectivity of a graph should be regarded as some kinds of chromatic numbers. However, they are different from classic chromatic numbers. These kinds of chromatic numbers
come from colorings by keeping some global structural properties of a graph, say connectivity; whereas the classic chromatic numbers come from colorings by keeping
some local structural properties of a graph, say adjacent vertices or edges. So, the employed methods to study them appear quite different sometimes. Of course,
local structural properties may yield global structural properties, and vice versa. But this is not always the case, say, local connectedness of a graph cannot guarantee
connectedness of the entire graph. So, many colored versions of connectivity parameters appeared in recent years, and we refer \cite{rainbow1, rainbow2, rainbow3, proper1, proper2, mono1, mono2} for
surveys.

Let $G$ be a graph that may have parallel edges but no loops. By deleting all parallel edges but one of them, we obtain a simple spanning subgraph of $G$, and call it the {\em underling graph} of $G$. If there are some parallel edges of an edge $e=ab$, then any monochromatic $ab$-cut contains $e$ and its parallel edges. Therefore, the following result is obvious, which means that we only need to think about simple graphs in
the sequel.
\begin{proposition}
Let $G'$ be the underling graph of a graph $G$. Then $md(G)=md(G')$.
\end{proposition}
The following result means that we only need to consider connected graphs in the sequel.
\begin{proposition}
If a simple graph $G$ has $t$ components $D_1,\cdots,D_t$, then $md(G)=\sum_{i\in[t]}md(D_i)$.
\end{proposition}

Let $G$ and $H$ be two graphs. The {\em union} of $G$ and $H$ is the graph $G\cup H$ with vertex set $V(G)\cup V(H)$ and edge set $E(G)\cup E(H)$. If $G$ and $H$ are vertex-disjoint, then let $G\vee H$ denote the {\em join} of $G$ and $H$, which is obtained from $G$ and $H$ by adding an edge between each vertex of $G$ and every vertex of $H$.

A block is {\em trivial} if it is a cut-edge. If $e=uv$ is an edge of $G$ with $d_G(v)=1$, we call $e$ a {\em pendent edge} of $G$ and $v$ a {\em pendent vertex} of $G$.

\section{Some basic results}

Let $G$ be a graph having at least two blocks. An edge-coloring of $G$ is an $MD$-coloring if and only if it is also an $MD$-coloring restricted on each block. Therefore, the following result is obvious.
\begin{proposition}\label{block}
If a connected graph $G$ has $r$ blocks $B_1,\cdots,B_r$, then $md(G)=\sum_{i\in [r]}md(B_i)$.
\end{proposition}
By above proposition, if $G$ is a tree, then $md(G)=n-1$.
\begin{proposition}\label{P1}
If $G$ is a cycle, then $md(G)=\lfloor\frac{|G|}{2}\rfloor$. Furthermore,
if $G$ is a unicycle graph with cycle $C$, then $md(G)=n-\lceil\frac{|C|}{2}\rceil$.
\end{proposition}

\begin{proof}
By Proposition \ref{block}, we only prove that $md(G)=\lfloor\frac{|G|}{2}\rfloor$ if $G$ is a cycle.

Let $G=C$ be a cycle. Suppose $C=v_1e_1v_2e_2\cdots v_{n-1}e_{n-1}v_ne_nv_1$. Let $r=\lfloor\frac{n}{2}\rfloor$. For $i\in[r]$ and $j\in[n]$, if $j\equiv i \pmod {r}$, then color $e_j$ by $i+1$. It is easy to verify that the edge-coloring of $C$ is an $MD$-coloring, and so $md(C)\geq r$.

Now we prove $md(C)\leq r$. If $md(C)\geq r+1$, there is an $MD$-coloring $\Gamma$ of $C$ such that $|\Gamma(C)|\geq r+1$. Then there exists a color $i$ of $\Gamma$ that colors only one edge $e$ of $C$, say $e=ab$. Because the monochromatic $ab$-cut must contain $e$ and some other edges of $C-e$, a contradiction.
\end{proof}

Because an $MD$-coloring of $G$ separates any two vertices by a monochromatic cut, it also separates any two vertices of a subgraph of $G$. So the following result is obvious.

\begin{proposition} \label{key}
Let $D$ be a subgraph of a graph $G$. If $\Gamma$ is an $MD$-coloring of $G$, then $\Gamma$ is also an $MD$-coloring restricted on $D$.
\end{proposition}

\begin{lemma}\label{subb}
Let $H$ be the union of graphs $H_1,\cdots,H_r$. If $\bigcap_{i\in[r]}E(H_i)\neq \emptyset$ and $md(H_i)=1$ for each $i\in[r]$, then $md(H)=1$.
\end{lemma}
\begin{proof}
We prove by contradiction. Suppose $\Gamma$ is an $MD$-coloring of $H$ with $|\Gamma(H)|\geq2$. Then there are two edges $e_1,e_2$ of $H$ such that $\Gamma(e_1)=1$ and $\Gamma(e_2)=2$.
W.l.o.g., let $e_1\in E(H_1)$ and $e_2\in E(H_2)$. Since $\Gamma$ is an $MD$-coloring restricted on $H_1$ (also $H_2$) and $md(H_1)=md(H_2)=1$, all edges of $H_1$ are
colored by $1$ and all edges of $H_2$ are colored by $2$ under $\Gamma$, which contradicts that $E(H_1)\cap E(H_2)\neq \emptyset$.
\end{proof}

\begin{lemma} \label{sub}
If $H$ is a connected spanning subgraph of $G$, then $md(H)\geq md(G)$.
\end{lemma}
\begin{proof}
 Let $H'$ be a graph obtained from $G$ by deleting an edge $e=ab$ where $e$ is in a cycle of $G$. If $md(H')\leq md(G)+1$, let $\Gamma$ be an extremal $MD$-coloring of $G$. Then $\Gamma$
 is an $MD$-coloring that is restricted on $H'$, and this implies that $e$ is the only edge of $G$ colored by $\Gamma(e)$. However, $e$ is in a cycle of $G$, and the monochromatic
 $ab$-cut has at least $2$ edges, a contradiction. Therefore, $md(H')\geq md(G)$.

If $H$ is a connected proper spanning subgraph of $G$, $H$ can be obtained from $G$ by deleting some edges in cycles one by one, consecutively. Therefore, the lemma is true.
\end{proof}
\begin{corollary}
For any connected graph $G$, $md(G)\leq n-1$, and the equality holds if and only if $G$ is a tree.
\end{corollary}

\begin{proof}
Since each connected graph has a spanning tree $T$, by Lemma \ref{sub} we have that $md(G)\leq md(T)=n-1$ if $G$ is connected. On the other hand, if $G$ is a connected  graph with $md(G)=n-1$ but $G$ is not a tree, then $G$ has a connected unicycle spanning subgraph $G'$. By Proposition \ref{sub} and \ref{P1}, $md(G)\leq md(G')<n-1$, a contradiction.
\end{proof}

\begin{lemma}\label{secq}
Let $G$ be a connected graph and $v\in V(G)$. If $v$ is neither a pendent vertex nor a cut-vertex of $G$, then $md(G)\leq md(G-v)$.
\end{lemma}
\begin{proof}
The following claim is useful for the proof of this lemma and for other proofs later.
\begin{claim}\label{G-v}
For any $MD$-coloring $\Gamma'$ of $G$, $\Gamma'(G)-\Gamma'(G-v)=\emptyset$.
\end{claim}
\begin{proof}
We proceed by contradiction. Let $e=vu$ be an edge of $ E(G)-E(G-v)$ and $\Gamma'(e)\notin \Gamma'(G-v)$. Since $d_G(v)\geq2$, there is another edge incident with $v$, say $f=vw$.
Because $v$ is not a cut-vertex, there is a cycle $C$ of $G$ with $E(C)- E(G-v)=\{e,f\}$. Because $\Gamma'$ is an $MD$-coloring restricted on $C$,
there are at least two edges in the monochromatic $uv$-cut of $C$ and the monochromatic $uv$-cut contains $e$. Since $\Gamma'(G)-\Gamma'(G-v)\neq\emptyset$,
$f$ is in the monochromatic $uv$-cut, i.e., $\Gamma'(e)=\Gamma'(f)$. Then, there is no monochromatic $uw$-cut in $C$, a contradiction.
\end{proof}

Let $\Gamma$ be an extremal $MD$-coloring of $G$. Then $\Gamma$ is an $MD$-coloring restricted on $G-v$. By Claim \ref{G-v}, $\Gamma(G)-\Gamma(G-v)=\emptyset$.
Therefore $md(G)=|\Gamma(G)|=|\Gamma(G-v)|\leq md(G-v)$.
\end{proof}
\begin{theorem}\label{n-2}
If $G$ is a $2$-connected graph, then $md(G)\leq \lfloor\frac{n}{2}\rfloor$.
\end{theorem}

\begin{proof}
Let $F=\{C,P_1,\cdots,P_t\}$ be an ear-decomposition of $G$ where $C$ is a cycle and $P_i$ is a path for $i\in[t]$.
The proof proceeds by induction on $|F|$. If $|F|=1$, then $G$ is a cycle, the theorem holds. If $|F|=t+1\geq2$, let $\Gamma$ be an extremal $MD$-coloring of $G$.
Then $\Gamma$ is an $MD$-coloring restricted on $G'$, where $G'$ is a graph obtained from $G$ by deleting $E(P_t)$ and the internal vertices of $P_t$.
By induction, we have
$$|\Gamma(G')|\leq md(G')\leq \lfloor\frac{|G'|}{2}\rfloor= \lfloor\frac{n-|P_t|+1}{2}\rfloor.$$

Suppose that the ends of $P_t$ are $a,b$ and $L$ is an $ab$-path of $G'$. Then $C'=L\cup P_t$ is a cycle of $G$. Because $\Gamma$ is an $MD$-coloring restricted on $C'$,
the monochromatic $ab$-cut contains at least one edge of $L$ and at least one edge of $P_t$, say $e$. Therefore, there are at most $|P_t|-1$ edges colored by $\Gamma(G)-\Gamma(G')$.
Since each color of $\Gamma(G)-\Gamma(G')$ colors at least two edges of $P_t-e$, then $|\Gamma(G)-\Gamma(G')|\leq \lfloor\frac{|P_t|-1}{2}\rfloor$. So,
$$md(G)=|\Gamma(G)|= |\Gamma(G')|+|\Gamma(G)-\Gamma(G')|\leq \lfloor\frac{n-|P_t|+1}{2}\rfloor+\lfloor\frac{|P_t|-1}{2}\rfloor\leq\lfloor\frac{n}{2}\rfloor.$$
\end{proof}

\section{Graphs with monochromatic disconnection number one}

In this section we will consider the monochromatic disconnection numbers for
some special graphs, such as triangular graphs (i.e., graphs with each of its edges in a triangle), complete multipartite graphs, chordal graphs,
square graphs and line graphs (the definitions of the last four graphs are as usual, we omit them). We denote the square graph and the line graph of
a graph $G$ by $G^2$ and $L(G)$, respectively.

For a graph $G$, we define a {\em relation} $\theta$ on the edge set $E(G)$ as follows: for two edges $e$ and $f$ of $G$, we say that $e\theta f$ if there exists a sequence of
subgraphs $G_1,\cdots,G_k$ of $G$ with $md(G_i)=1$ for any $i\in[k]$,
such that $e\in G_1$ and $e'\in G_k$, and $|V(G_i)\cap V(G_{i+1})|\geq2$ for $i\in[k-1]$.
It is easy to check that $\theta$ is {\em symmetric, reflexive} and {\em transitive} and therefore an equivalent relation on $E(G)$. We call a graph $G$ a {\em closure} if $e\theta e'$ for any two edges $e,e'$ of $E(G)$.

\begin{lemma} \label{eq-c-1}
If a graph $G$ is a closure, then $md(G)=1$.
\end{lemma}
\begin{proof}
Suppose $md(G)\geq2$ and $\Gamma$ is an extremal $MD$-coloring of $G$. Then there exist two edges, say $e,e'$, of $G$, such that $\Gamma(e_1)\neq \Gamma(e_2)$. Since $G$ is a closure, there is a sequence of subgraphs $G_1,\cdots,G_k$ with $md(G_i)=1$ for any $i\in[k]$,
such that $e\in G_1$ and $e'\in G_k$, and $G_i$ and $G_{i+1}$ have at least two common vertices, say $a_i,b_i$, for $i\in[k-1]$.
Since all edges of each $G_i$ must be colored with a same color
under $\Gamma$, $\Gamma(G_i)=\Gamma(G_{i-1})$. Otherwise there is no monochromatic $a_ib_i$-cut. Therefore, $\Gamma(e)=\Gamma(G_1)=\Gamma(G_2)=\cdots=\Gamma(G_{k})=\Gamma(e')$, a contradiction. So, $md(G)=1$.
\end{proof}

\begin{theorem} \label{MDC-special}
If $G$ is one of the following graphs, then $md(G)=1$.
\begin{enumerate}
\item $G=H\vee v$ where $H$ is a connected graph and $v$ is an additional vertex;
\item $G$ is a multipartite graph other than $K_{1,n-1}$ and $K_{2,2}$;
\item $G$ is a $2$-connected chordal graph;
\item $G=H^2$ where $H$ is a connected graph;
\item $G=L(H)$ where $H$ is a connected triangular graph.
\end{enumerate}
\end{theorem}
\begin{proof}
{\bf (1)} \ If $H=K_1$, the result holds. If $|H|\geq2$, let $T$ be a spanning tree of $G$ and $u$ be a leaf of $T$. By induction, $md((T-u)\vee v)=1$. Since $V(T\vee v)-V((T-u)\vee v)=\{u\}$ and $u$ is neither a pendent vertex nor a cut vertex of $T\vee v$, by Lemma \ref{secq}, $md(T\vee v)\leq md((T-u)\vee v)$. Since $T\vee v$ is a connected spanning subgraph of $H\vee v$, by Lemma \ref{sub}, $md(H\vee v)\leq md(T\vee v)$. Therefore, $md(T\vee v)=1$.

{\bf (2)} \ We first show that $md(K_{2,3})=1$. Any $MD$-coloring of $C_4$ can have only two cases, one is trivial and the other is to assign colors $1,2$ to the four edges of $C_4$ alternately.
Let $H=K_{2,3}$ and the bipartition of $H$ be $A=\{a,c\}$ and $B=\{b,d,u\}$. If $md(H)\geq 2$, there is an $MD$-coloring $\Gamma$ of $K_{2,3}$ with $|\Gamma(H)|\geq2$. Therefore, at least one of the three $4$-cycles of $H$ has a nontrivial $MD$-coloring. Let the three $4$-cycles of $H$ be $H_1=H[a,b,c,d], H_2=H[a,b,c,u]$ and $H_3=H[a,d,c,u]$. By symmetry, suppose that $H_1$ is colored nontrivially, say $\Gamma(ad)=\Gamma(bc)=1$ and $\Gamma(ab)=\Gamma(cd)=2$.
Then $\Gamma$ is a nontrivial $MD$-coloring restricted on $H_2$ with $\Gamma(au)=1$ and $\Gamma(cu)=2$. It is obvious that $\Gamma$ is not an $MD$-coloring restricted on $H_3$, which contradicts that
$\Gamma$ is an $MD$-coloring of $G$. Therefore, $md(H)=1$.

Let $G$ be a complete bipartite graph other than $K_{1,n-1}$ and $K_{2,2}$. Suppose that $A,B$ are the bipartition of $G$ with $A=\{u,v,a_1,\cdots,a_s\}$ and $B=\{u',v',b_1,\cdots,b_t\}$.
Then at least one of $s,t$ is not zero. Let $G_i=G[u,u',v,v',a_i]$ and $G'_j=G[u,u'v,v',b_j]$ for $i\in[s]$ and $j\in[t]$. Since each $G_i$ or $G'_j$ is $K_{2,3}$ and
$\bigcap_{i\in [s]}E(G_i)\cap\bigcap_{j\in[t]}E(G'_j)=E(G[u,u',v,v'])$, by Lemma \ref{subb} we have $md(\bigcup_{i\in [s]}G_i\cup\bigcup_{j\in[t]}G'_j)=1$. Since $\bigcup_{i\in [s]}G_i\cup\bigcup_{j\in[t]}G'_j$ is a connected spanning subgraph of $G$, then $md(G)=1$.

Let $G=G_1$ be a complete $r$-partite graph with $r\geq3$ and let $V=\{v_1,\cdots,v_t\}$ be one part of $G$. Let $G_i=G-\{v_{1},\cdots,v_{i-1}\}$ for $i\in \{2,\cdots,t\}$. Then each $v_i$ is neither a pendent vertex nor a cut vertex of $G_i$, by Lemma \ref{secq}, $md(G_1)\leq md(G_2)\leq\cdots\leq md(G_t)$. However, $G_t=(G-V)\vee v_t$ and $G-V$ is a connected graph, which implies $md(G_t)=1$. Therefore, $md(G)=1$.

{\bf (3)} \ A {\em simplicial order }of a graph $H$ is an enumeration $v_1,\cdots,v_n$ of its vertices such that the neighbors of $v_i$ in $H[\{v_i,\cdots,v_n\}]$ induce a complete graph. A graph is chordal if and only if it has a simplicial order (see Corollary 9.22 on page 273 of \cite{B}). Suppose that a simplicial order of $G$ is $u_1,\cdots,u_n$ and $G_i=G[\{u_i,\cdots,u_n\}]$ for $i\in[n-2]$ (then $G=G_1$). Let $H_i=G_i[N_{G_i}(u_i)]$. Since $G$ is  $2$-connected, each $H_i$ is a complete graph other than $K_1$ and thus $G_{n-1}=G[v_{n-1},v_n]$ is a $K_2$. Therefore, $v_i$ is neither a pendent vertex nor a cut vertex of $G_i$ for $i\in[n-2]$, and hence by Lemma \ref{secq}, $md(G_i)\leq md(G_{i+1})$. So, $md(G)\leq md(G_{n-1})=1$.

{\bf (4)} \ The result holds for $G=K_2$. We prove it by induction on $|G|$. If $|G|\geq3$, let $T$ be a spanning tree of $G$ and $v$ be a leaf of $T$. Then $T^2-v=(T-v)^2$. Since $v$ is neither a pendent vertex nor a cut vertex of $T^2$, then $md(T^2)\leq md((T-v)^2)=1$. Since $T^2$ is a spanning tree of $G^2$, then $md(G^2)\leq md(T^2)$. Therefore, $md(G^2)=1$.

{\bf (5)} \ Let $A$, $B$ be two edge-induced subgraphs of $G$. We define
$$d_G(A,B)=\min\{d_G(u,v):~u\in V(A),~v\in V(B)\}.$$

Because the line graph of a triangular graph is also a triangular graph, we only need to show that $L(G)$ is a closure, i.e., we need to show that
for every two edges $l_1,l_2$ of $L(G)$, $l_1\theta l_2$. For each edge $e_i$ of $G$, we denote the corresponding vertex of $L(G)$ by $u_i$. We proceed by induction on $d_{L(G)}(l_1,l_2)$.

If $d_{L(G)}(l_1,l_2)=0$, this implies that $l_1$ and $l_2$ has a common vertex.
Let $l_1=u_1u_2$ and $l_2=u_2u_3$. If $G[e_1,e_2,e_3]$ is a triangle (denote it by $\Delta$) of $G$, then $L(\Delta)$ is a triangle of $L(G)$ containing $l_1,l_2$, and so $l_1\theta l_2$;
if just two edges of $e_1,e_2,e_3$ are in a triangle $\Delta$ of $G$, suppose $\Delta=G[e_1,e_2,e_4]$.
Then $G[e_2,e_3,e_4]$ is a star (call the star $S$). Because $L(\Delta)$ and $L(S)$ are two triangles of $L(G)$ and they have a common edge $u_2u_4$, and because $L(\Delta)$ contains $l_1$ and $L(S)$ contains $l_2$,
then $l_1\theta l_2$; if none of triangles of $G$ contains at least two of $e_1,e_2,e_3$, suppose $\Delta=G[e_4,e_2,e_5]$ is a triangle of $G$ where $e_4$ is adjacent to $e_3$ and $e_5$ is adjacent to $e_1$. Then $S_1=G[e_1,e_2,e_5]$ and $S_2=G[e_3,e_2,e_4]$ are two stars of $G$. Therefore, $L(S_1)$, $L(\Delta)$ and $L(S_2)$ are three triangles of $L(G)$ such that $L(S_1), L(\Delta)$ have a common edge $u_2u_5$ and $L(S_2), L(\Delta)$ have a common edge $u_2u_4$. So, $l_1\theta l_2$.

If $d_{L(G)}(l_1,l_2)=r>0$, let $l_1=u_1u_2$ and $l_2=u_3u_4$. Suppose $P$ is a shortest path of $L(G)$ connecting $l_1$ and $l_2$. Then $|P|=r$. W.l.o.g., suppose $l_3=u_3u_5$ is a pendent edge of $P$.
Then $d_{L(G)}(l_1,l_3)=r-1$ and $d_{L(G)}(l_2,l_3)=0$. By induction, $l_1\theta l_3$ and $l_2\theta l_3$. Therefore, $l_1\theta l_2$.
\end{proof}

{\bf Remark 1:} By Theorem \ref{MDC-special} (2), $md(K_n)=1$ for $n\geq2$. Let $v$ be a minimum degree vertex of $K_n^-$ ($n\geq4$). Then $K_n^--v=K_{n-1}$. Since $v$ is neither a pendent vertex nor a cut vertex of $K_n^-$, $md(K_n^-)\leq md(K_{n-1})=1$, i.e., $md(K_n^-)=1$ for $n\geq4$.
$\blacksquare$

As we have seen that a lot of graphs have the monochromatic disconnection number equal to 1. We may guess that the following result holds and it does hold actually.
\begin{theorem}
For almost all graphs $G$, $md(G)=1$ holds.
\end{theorem}
\begin{proof}
Let $G\sim \mathcal{G}_{n,\frac{1}{2}}$, that is, $G$ is a random graph on $n$ vertices chosen by picking each pair of vertices as an edge randomly and independently with probability $\frac{1}{2}$.
Let $\mathcal{A}_{uv}$ be the set of events that $u$ and $v$ have at most $2$ common neighbors and $\mathcal{A}=\bigcup_{u,v\in V(G)}\mathcal{A}_{uv}$. Let $\mathcal{A}^i_{uv}$ be the set of events that $u,v$ have exactly $i$ common neighbors. Then $\mathcal{A}_{uv}=\bigcup^2_{i=0}\mathcal{A}^i_{uv}$.
For a vertex $w$ of $V(G)-\{u,v\}$, since
$$Pr[w\mbox{ is a common neighbor of }u\mbox{ and }v]=\frac{1}{4}$$ and
$$Pr[w\mbox{ is not a common neighbor of }u\mbox{ and }v]=\frac{3}{4},$$
then $$Pr[\mathcal{A}^i_{uv}]={n-2\choose i}(\frac{1}{4})^{i}(\frac{3}{4})^{n-i-2}.$$
Therefore, $Pr[\mathcal{A}_{uv}]< 3n^2(\frac{3}{4})^{n-4}$ when $n$ is large enough, and then
$$Pr[\mathcal{A}]\leq {n\choose 2}Pr[\mathcal{A}_{uv}]< 3n^4 (\frac{3}{4})^{n-2}\rightarrow 0 \mbox{ as } n\rightarrow\infty.$$
This implies that almost all graphs have property that any two vertices have at least $3$ common neighbors.
We will complete the proof by showing that $md(G)=1$ if a graph $G$ has the property that every two vertices of $G$ have at least three common neighbors.

For any two edges $e=ab$ and $f=uv$ of $G$, there is a path $P$ of $G$ such that the pendent edges of $P$ are $e$ and $f$.
Let $e_1=x_1x_2$ and $e_2=x_2x_3$ be two adjacent edges of $P$. Then $x_1$ and $x_3$ have three common neighbors ($x_2$ is one of them) and
thus $e_1$ and $e_2$ are in a $K_{2,3}$ of $G$. This implies $e_1\theta e_2$. By transitivity, $e\theta f$. Therefore, $G$ is a closure,
and so $md(G)=1$ by Lemma \ref{eq-c-1}.
\end{proof}

\section{Nordhaus-Gaddum-type results}

For a graph parameter, it is always interesting to get the Nordhaus-Gaddum-type results, see \cite{AH} and \cite{Chang, HH, WU, Mao, NG, Shan, Zhang} for more such results on various
kinds of graph parameters. This section is devoted to get the Nordhaus-Gaddum-type results for our parameter $md(G)$.

For a connected graph $G$, a vertex $v$ is {\em deletable} if $G-v$ is connected. Let $\mathcal{B}$ be the set of blocks of $G$ and $S$ be the set of cut-vertices of $G$. A {\em block tree} of $G$ is a bipartite graph $B(G)$ with  bipartition $\mathcal{B}$ and $S$, and a block $B$ has an edge with a cut-vertex $v$ in $B(G)$ if and only if $B$ contains $v$. Therefore, every pendent vertex of $B(G)$ is a block (call it {\em leaf-block}).

Because $B(G)$ is a tree, there are at least two leaves in $B(G)$ if $G$ has more than one block.
For a leaf-block $B$ of $G$, there are $|B|-1$ deletable vertices in the block. This implies that every graph has at least two deletable vertices.

\begin{fact}\label{111}
If $G$ is a connected simple graph with $|G|\geq 2$, then $G$ has at least two deletable vertices. Furthermore, $G$ has exactly two deletable vertices if and only if $G$ is a path.
\end{fact}
\begin{proof}
We only need to deal with the case that $G$ is not a path. If $B(G)$ has at least three leaves, or $B(G)$ has two leaves with one being nontrivial, then $G$ has at least three deletable vertices; if $B(G)$ has exactly two trivial leaf-blocks, because $G$ is not a path, there is a nontrivial block $B$ and $B$ has exactly two cut vertices, then $B$ has at least $|B|-2\geq1$ deletable vertices.
Therefore $G$ has at least three deletable vertices.
\end{proof}

Because a complete graph $K_n$ can be decomposed into two connected graphs if and only if $n\geq4$, in this section we always assume $n\geq4$.

\begin{lemma} \label{dele}
Suppose $G$ and $\overline{G}$ are connected spanning subgraphs of $K_n$ with $n\geq 5$. Then there is a vertex $v$ of $K_n$ such that $v$ is deletable for both $G$ and $\overline{G}$.
\end{lemma}

\begin{proof}
If both $G$ and $\overline{G}$ are $2$-connected spanning subgraphs of $K_n$, then every vertex is deletable for both $G$ and $\overline{G}$. So, we assume that at least one of $G_1=G$ and $G_2=\overline{G}$ has
cut vertices. Let $v$ be a cut vertex of $G_1$ and let $S_1,\cdots,S_r$ be the components of $G_1-v$. Then $F_i=G_1[v\cup S_i]$ is a connected graph. It is obvious that $G_2-v$ contains a
complete $r$-partite spanning subgraph, denote it by $H$. W.l.o.g., let $e=vu_1$ be an edge of $G_2$ and $u_1\in S_1$. We distinguish the following cases to discuss.

{\em Case 1}: $r=2$ and $|S_2|\geq2$, or $r\geq3$.

By Fact \ref{111}, there is a deletable vertex $u_2$ of $F_2$ and $u_2\neq v$. Then $G_1-u_2$ is connected. If $r=2$ and $|S_2|\geq2$, because $H$ is a complete bipartite graph with $|S_2|\geq2$, and $vu_1$ is an edge of $G_2$ with $u_1\in S_1$, then $G_2-u_2$ is connected; if $r\geq3$, then $G_2-u_2$ is also connected. Therefore, $u_2$ is deletable for both $G$ and $\overline{G}$.

{\em Case 2}: $r=2$ and $|S_2|=1$. Let $S_2=\{u_2\}$.

If $F_1$ is not a path, by Fact \ref{111}, $F_1$ has a deletable vertex $w$ different from $v$ and $u_1$. Then $G_1-w$ is connected. Because $u_2$ connects to all vertices of $S_1$ and $vu_1$ is not affected in $G_2-w$, then $G_2-w$ is also connected. Therefore, $w$ is deletable for both $G$ and $\overline{G}$.

If $F_1$ is a path, then suppose $y$ is a leaf of $F_1$ other than $v$ and $x$ connects $y$ in $F_1$. Because $n\geq5$, $v$ is not connected to both $x$ and $y$ in $F_1$. Therefore, $vx$ and $vy$ are edges of $G_2$, both $G_1-y$ and $G_2-y$ are connected. Therefore, $y$ is deletable for both $G$ and $\overline{G}$.
\end{proof}

\begin{theorem}\label{N-G-T}
Suppose $G$ and $\overline{G}$ are connected spanning subgraphs of $K_n$. Then $md(G)+md(\overline{G})\leq n+1$ for $n\geq5$, and $md(G)+md(\overline{G})\geq2$ for $n\geq8$. Furthermore, the upper bound and the lower bound are sharp.
\end{theorem}

\begin{proof}
Because both $G$ and $\overline{G}$ are non-empty graphs, then $md(G)+md(\overline{G})\geq2$ is obvious for $n\geq8$. So, we need to show that $md(G)+md(\overline{G})\leq n+1$ for $n\geq5$.

If $n=5$, there are five cases to consider for the graphs $G$ and $\overline{G}$, and all of the five cases imply that $md(G)+md(\overline{G})\leq6= n+1$ (see Figure \ref{ngt-2}).

We proceed by induction on $n$. The theorem holds for $n=5$. If $n>5$, by Lemma \ref{dele} there is a deletable vertex $v$ for both $G$ and $\overline{G}$. Let $G'=G-v$.
Then $G'$ and $\overline{G'}$ are connected subgraphs of $K_{n-1}$. By induction, $md(G')+md(\overline{G'})\leq n$. Let $\Gamma$ be an extremal $MD$-coloring of $G$.

Because $n> 5$, at least one of $d_G(v)$ and $d_{\overline{G}}(v)$ is greater than 1 (say $d_G(v)=r \geq2$). Then $v$ is neither a pendent vertex nor a cut vertex of $G$, and so by Lemma \ref{secq}, $md(G)\leq md(G')$.
If $d_{\overline{G}}(v)\geq2$, we also have $md(\overline{G})\leq md(\overline{G'})$; if $d_{\overline{G}}(v)=1$, then $md(\overline{G})=md(\overline{G'})+1$. Therefore,
$md(G)+md(\overline{G})\leq md(G')+md(\overline{G'})+1\leq n+1$.

Now we show that the upper bound is sharp for $n\geq5$. Let $B_n$ be a tree with $|B_n|=n$ and $\Delta(B_n)=n-2$. Then $\overline{B_n}$ is a graph obtained by joining a pendent edge to one of the vertices of $K_{n-1}^-$ with minimum degree. Since $G$ and $\overline{G}$ are connected graphs and $md(G)=n-1,md(\overline{G})=2$, then $md(B_n)+md(\overline{B_n})=n+1$.

We now show that the lower bound is sharp for $n\geq8$. Let $V(K_n)=A\cup B\cup\{a,b,u,v\}$ where both $|A|,|B|$ are greater than 1. Let $J$ be a complete bipartite graph with bipartition $A\cup\{a,u\}$ and $B\cup\{b,v\}$. Then $C=J[a,b,u,v]$ is a $C_4$. Let $G$ be a graph obtained from $J$ by deleting the edges of $C$. Let $G_a=J-\{b,u,v\}$, $G_b=J-\{a,u,v\}$, $G_u=J-\{a,b,v\}$ and $G_v=J-\{a,b,u\}$. Then $G$ is the union of $G_a,G_b,G_u$ and $G_v$.
Because $G_a,G_b,G_u$ and $G_v$ are complete bipartite graphs other than $K_{2,2}$ and stars, by Theorem \ref{MDC-special} $(2)$, we have $md(G_a)=md(G_b)=md(G_u)=md(G_v)=1$. Thus, by Lemma \ref{subb}, $md(G)=1$ (see Figure \ref{ngt-l}). For $\overline{G}$, since $H_1=\overline{G}[A\cup a\cup u]$, $H_2=\overline{G}[a,b,u,v]$ and $H_3=\overline{G}[B\cup b\cup v]$ are complete graphs, and $E(H_i)\cap E(H_{i+1})\neq \emptyset$ for $i=1,2$, $md(\overline{G})=1$ (see Figure \ref{ngt-l}). Therefore, the lower bound is sharp for $n\geq8$.
\end{proof}

\begin{figure}[h]
    \centering
    \includegraphics[width=260pt]{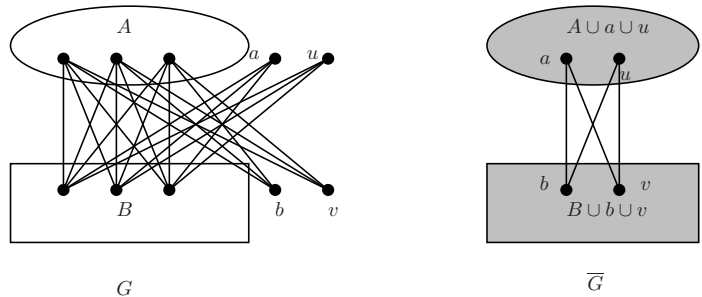}\\
    \caption{Extremal graphs for $md(G)+md(\overline{G})=2$ when $|G|\geq 5$.} \label{ngt-l}
\end{figure}

{\bf Remark 2:} By Theorem \ref{N-G-T}, the lower bound of $md(G)+md(\overline{G})$ for $4\leq n\leq7$ and the upper bound of $md(G)+md(\overline{G})$ for $n=4$ are not considered. We will discuss them below.

{\bf({\uppercase\expandafter{\romannumeral1}}) \ }  For $n=4$, because $K_4$ can only be decomposed into two $P_3$, then $md(G)+md(\overline{G})=6= n+2$.

\begin{figure}[h]
    \centering
    \includegraphics[width=260pt]{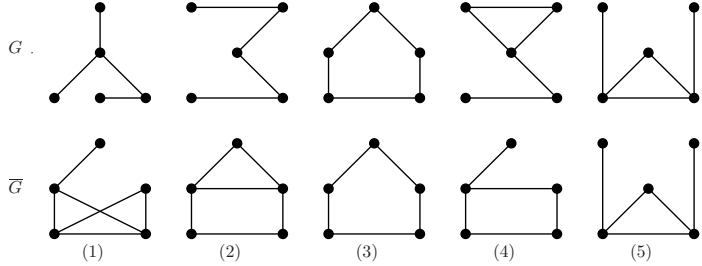}\\
    \caption{The five cases of $G$ and $\overline{G}$ we talk about when $n=5$.} \label{ngt-2}
\end{figure}

{\bf({\uppercase\expandafter{\romannumeral2}}) \ }  For $n=5$, there are ten cases for $G$ and $\overline{G}$. However, by symmetry, we only need to discuss the five cases depicted in Figure \ref{ngt-2}. Among all the five cases, $(3)$ implies that the lower bound of $md(G)+md(\overline{G})$ is $4$.

{\bf({\uppercase\expandafter{\romannumeral3}})\ } For $n=6$, $e(K_6)=15$. Because $G$ and $\overline{G}$ are connected spanning subgraphs of $K_6$, both $e(G)$ and $e(\overline{G})$ are greater than or equal to $5$.

If $e(G)=5$ and $e(\overline{G})=10$, then $md(G)+md(\overline{G})\geq6$.

If $e(G)=6$ and $e(\overline{G})=9$, then $G$ is a unicycle graph and the length of the cycle is at most $6$. By Proposition \ref{P1}, we have $md(G)\geq 3$. So, $md(G)+md(\overline{G})\geq4$.

If $e(G)=7$ and $e(\overline{G})=8$, we assume that $G$ has $t$ blocks. If $t\geq3$, by proposition \ref{block} we have $md(G)\geq3$. Thus, $md(G)+md(\overline{G})\geq4$.
\begin{figure}[h]
    \centering
    \includegraphics[width=250pt]{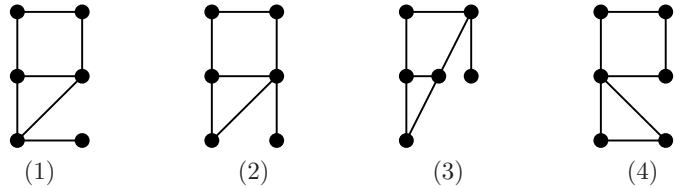}\\
    \caption{The four cases of graph $G$ when $t=2$.} \label{ngt-31}
\end{figure}
If $t=2$, $G$ is isomorphic to one of the four graphs in Figure \ref{ngt-31}. Because every graph $F$ in Figure \ref{ngt-31} has $md(F)=3$, then $md(G)+md(\overline{G})\geq4$.
If $t=1$, there are three cases to consider (see Figure \ref{ngt-32}). As shown in Figure \ref{ngt-32}, we give an extremal $MD$-coloring for each graph. Because the last two cases of Figure \ref{ngt-32} imply that $md(G)+md(\overline{G})=4$, the lower bound of $md(G)+md(\overline{G})$ is $4$.

\begin{figure}[h]
    \centering
    \includegraphics[width=250pt]{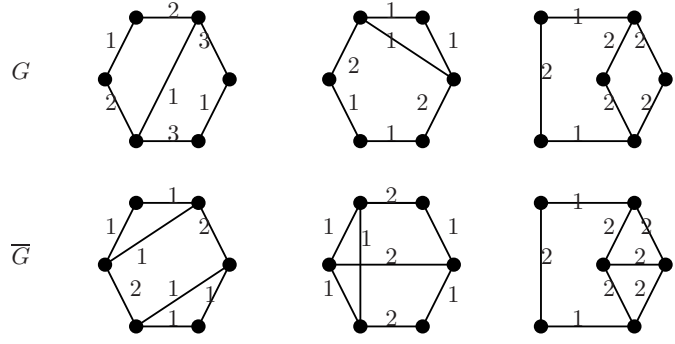}\\
    \caption{The three cases of $G$ and $\overline{G}$ when $t=1$.} \label{ngt-32}
\end{figure}

{\bf({\uppercase\expandafter{\romannumeral4}}) \ } For $n=7$, the lower bound of $md(G)+md(\overline{G})$ is $2$. In fact, we only need to construct a graph $G$ (see in Figure \ref{ngt-4} $(1)$) and $\overline{G}$ (see in Figure \ref{ngt-4} $(2)$) such that $md(G)=md(\overline{G})=1$.

Let $\Gamma$ be an extremal $MD$-coloring of $G$. Because $G_1=G[g,f,c,d,h]$ and $G_2=G[g,f,c,d,b]$ are isomorphic to $K_{2,3}$, and because $G_1$ and $G_2$ have common edges, then by Lemma \ref{subb} we have $md(G_1\cup G_2)=1$. Because $a$ is neither a cut vertex nor a pendent vertex of $G$, then $md(G)\leq md(G-a)=1$.

\begin{figure}[h]
    \centering
    \includegraphics[width=280pt]{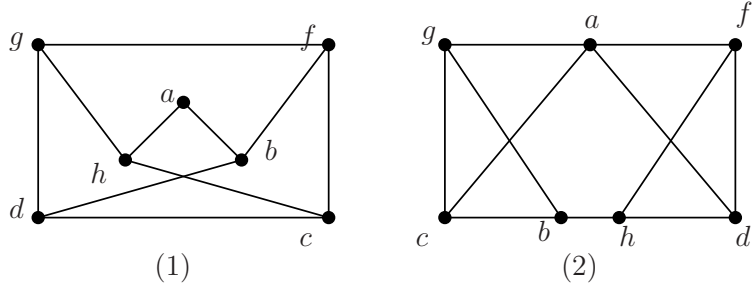}\\
    \caption{Extremal graphs for $n=7$.} \label{ngt-4}
\end{figure}

Because $\overline{G}[g,b,c,a]$ and $\overline{G}[a,f,h,d]$ are isomorphic to $K_4^-$, all edges of $\overline{G}[g,b,c,a]$ ($\overline{G}[a,f,h,d]$) are colored the same under any $MD$-coloring of $\overline{G}$. Then the $5$-cycle $\overline{G}[g,a,f,h,b]$ just has one trivial $MD$-coloring. Therefore, $md(\overline{G})=1$.
 $\blacksquare$

For ease of reading, the lower bounds and upper bounds of $md(G)+md(\overline{G})$ for $n\geq4$ are summarized in the following table.

\begin{table}[!htbp]
\centering
\begin{tabular}{|c|c|c|c|c|}
\hline
  & $n=4$  & $n=5$ & $n=6$& $n\geq7$ \\
 \hline
 Lower bound &$6$ & $4$&$4$&$2$ \\
 \hline
 Upper bound & $6$ & $6$ &$7$& $n+1$ \\
 \hline
\end{tabular}
\caption{The bounds of $md(G)+md(\overline{G}$).}
\end{table}

\begin{theorem}
If both $G$ and $\overline{G}$ are connected and $|G|=n\geq4$, then $md(G)\cdot md(\overline{G})=9$ for $n=4$; $4\leq md(G)\cdot md(\overline{G})\leq9$ for $n=5$;
$3\leq md(G)\cdot md(\overline{G})=10$ for $n=6$ and $1\leq md(G)\cdot md(\overline{G})\leq2(n-1)$ for $n\geq7$. Furthermore, the bounds are sharp.
\end{theorem}
\begin{proof}
We first show the upper bounds.

If $n=4$, then $G=\overline{G}=P_3$, and so $md(G)\cdot md(\overline{G})=9$; if $n=5$, then because $md(G)+md(\overline{G})=6$, we have $md(G)\cdot md(\overline{G})\leq9$.
The graphs $G$ and $\overline{G}$ are shown in Figure \ref{ngt-2} $(4)$, implying that $md(G)\cdot md(\overline{G})=9$.

We will show the upper bounds for $n\geq6$. The proof proceeds by induction on $n$. We will show the inductive base $n=6$ and the inductive step $n>6$ simultaneously.
Let $G$ and $\overline{G}$ be connected graphs with $n\geq6$. By Lemma \ref{dele}, there is a vertex $v$ such that both $G-v$ and $\overline{G}-v$ are connected.

{\em Case 1.} \ $d_G(v)\geq2$ and $d_{\overline{G}}(v)\geq2$.

Then $v$ is not a pendent vertex or a cut vertex of $G$ and $\overline{G}$. By Lemma \ref{secq}, $md(G)\leq md(G-v)$ and $md(\overline{G})\leq md(\overline{G}-v)$. Therefore, $md(G)\cdot md(\overline{G})\leq md(G-v)\cdot md(\overline{G}-v)$. If $n=6$, $md(G)\cdot md(\overline{G})\leq 9< 2(n-1)$; if $n>6$, by induction on $n$, $md(G)\cdot md(\overline{G})\leq md(G-v)\cdot md(\overline{G}-v)\leq 2(n-2)<2(n-1)$.

{\em Case 2.} \ $d_G(v)=1$ and $d_{\overline{G}}(v)=n-2$.

Let $u$ be the neighbor of $v$ in $G$. Then, $v$ connects every vertex of $V(\overline{G})-\{u,v\}$ in $\overline{G}$.

If $u$ is not a cut-vertex of $\overline{G}-v$, then $\overline{G}-u=v\vee (\overline{G}-\{u,v\})$ and thus $md(\overline{G}-u)=1$. Therefore, $md(\overline{G})=1$.

If $\overline{G}-\{u,v\}$ has two components $S_1$ and $S_2$, then $\overline{G}-u=(v\vee S_1)\cup (v\vee S_2)$. Since $md(v\vee S_1)=md(v\vee S_2)=1$, then $md(\overline{G}-u)=2$. Since $u$ is not a pendent vertex or cut vertex of $\overline{G}$, by Lemma \ref{secq}, $md(\overline{G})\leq md(\overline{G}-u)\leq2$.

If $\overline{G}-\{u,v\}$ has components $S_1,\cdots,S_k$ where $k\geq 3$, then let $w_i$ be a vertex connects $u$ in $S_i$ for $i\in[k]$.
Then $md(v\vee S_i)=1$ for $i\in[k]$. We now show $md(\overline{G})=1$. Otherwise, there is an $MD$-coloring $\Gamma$ of $G$ with $|\Gamma(G)|\geq2$.
Since $u$ is not a pendent vertex or a cut vertex of $\overline{G}$, by Claim \ref{G-v}, $\Gamma(\overline{G}-u)=\Gamma(\overline{G})$. Then there are two edges $e_1$ and $e_2$ of $\overline{G}-u$ such that $\Gamma(e_1)\neq\Gamma(e_2)$.
Since $md(v\vee S_i)=1$ for $i\in[k]$, w.l.o.g., let $e_1=vw_1$ and $e_2=vw_2$, then $G[u,v,w_1,w_2,w_3]\cong K_{2,3}$. This contradicts that $\Gamma$ is an $MD$-coloring restricted on the subgraph $G[u,v,w_1,w_2,w_3]$. Therefore,  $md(\overline{G})=1$.

According to the above, $md(\overline{G})\leq2$. Since $md(G)\leq n-1$, then $md(G)\cdot md(\overline{G})\leq 2(n-1)$ for $n\geq6$.

The graphs $B_n$ and $\overline{B_n}$ defined in the proof of Theorem \ref{N-G-T} show that $md(B_n)\cdot md(\overline{B_n})=2(n-1)$. So, the upper bound is sharp for $n\geq6$.

Now we show the lower bounds.

If $n=4$, $md(G)\cdot md(\overline{G})=9$; if $n\geq7$, since there are graphs $G$ and $\overline{G}$ such that $md(G)+md(\overline{G})=2$, then $md(G)\cdot md(\overline{G})=1$, i.e., the lower bound is sharp.

If $n=5$, $md(G)\cdot md(\overline{G})$ is minimum when $G$ and $\overline{G}$ are graphs shown in Figure \ref{ngt-2} $(3)$, which implies that $md(G)\cdot md(\overline{G})=4$.

If $n=6$, since $md(G)+md(\overline{G})\geq4$, $md(G)\cdot md(\overline{G})\geq3$. Let $G$ be a graph obtained by connecting an additional vertex $w$ to a vertex $u$ of a $5$-cycle (which implies $md(G)=3$). Then $\overline{G}$ is a graph obtained by connecting $w$ to every vertex of $\overline{C_5}$ except for $u$. Then $u$ is neither a pendent vertex nor a cut vertex of $\overline{G}$, $md(\overline{G})\leq md(\overline{G}-u)$. Since $\overline{G}-\{w,u\}$ is a path and $\overline{G}-u=v\vee(\overline{G}-\{w,u\})$, then $md(\overline{G}-u)=1$.  Therefore, $md(\overline{G})=1$, the lower bound is sharp for $n=6$.
\end{proof}
For ease of reading, the lower bounds and upper bounds of $md(G)\cdot md(\overline{G})$ for $n\geq4$ are summarized in the following table.
\begin{table}[!htbp]
\centering
\begin{tabular}{|c|c|c|c|c|}
\hline
  & $n=4$  & $n=5$ & $n=6$& $n\geq7$ \\
 \hline
 Lower bound &$9$ & $4$&$3$&$1$ \\
 \hline
 Upper bound & $9$ & $9$ &$10$& $2(n-1)$ \\
 \hline
\end{tabular}
\caption{The bounds of $md(G)\cdot md(\overline{G}$).}
\end{table}

\end{document}